\documentclass[12pt]{amsart}

\usepackage{amsmath,amssymb,amsfonts,graphics,amsthm}
\usepackage{verbatim}
\usepackage{fullpage}
\usepackage{tikz}

\newcommand{\N}{\mathbb{N}}

\newcommand{\C}{\mathbb{C}}
\newcommand{\F}{\mathbb{F}}
\newcommand{\G}{\mathbb{G}}

\newcommand{\T}{\mathbb{T}}

\newcommand{\mc}{\mathcal}
\newcommand{\id}{\text{id}}

\newcommand{\TL}{\text{TL}}

\newcommand{\Pol}{\text{Pol}}

\newtheorem{thm}{Theorem}[section]
\newtheorem{cor}[thm]{Corollary}
\newtheorem{lem}[thm]{Lemma}
\newtheorem{prop}[thm]{Proposition}

\theoremstyle{definition}
\newtheorem{defn}[thm]{Definition}

\theoremstyle{remark}
\newtheorem{rem}[thm]{Remark}

\numberwithin{equation}{section}

\begin{document}

\title[Strong Convergence]
{Strong asymptotic freeness for free orthogonal quantum groups} 

\date{\today}

\author{Michael Brannan}

\address{Michael Brannan: Department of Mathematics, University  of Illinois at Urbana-Champaign, Urbana, IL 61801, USA}
\email{mbrannan@illinois.edu}
\urladdr{http://www.math.uiuc.edu/~mbrannan}

\thanks{This research was supported by an NSERC Postdoctoral Fellowship}

\begin{abstract}
We prove that the normalized standard generators of the free orthogonal quantum group $O_N^+$ converge strongly to a free  semicircular system as $N \to \infty$.  Analogous results are obtained for the free unitary quantum groups, and some applications are given.
\end{abstract}
\maketitle

\section{Introduction} \label{section:intro}

Given a closed subgroup $G$ of the compact Lie group $U_N$ of $N \times N$ unitary matrices over $\C$, a fundamental mathematical problem is the computation of \textit{polynomial integrals} over $G$.  That is, we consider the $\ast$-algebra $\Pol(G) \subseteq C(G)$ of polynomial functions on $G$ generated by the $N^2$ standard coordinate functions \[\{u_{ij}\}_{1 \le i,j \le N} \subset C(G); \qquad u_{ij}(g) = (i,j)\text{-th coordinate of the matrix } g \in G, \] and seek to evaluate the integrals 
\[h_G(f) = \int_G f(g)dg \qquad (f \in \Pol(G)),\] where $dg$ denotes the Haar probability measure on $G$.  Equivalently, this amounts to determining the joint distribution of the standard coordinates $\{u_{ij}\}_{1 \le i,j \le N}$, viewed as bounded random variables in $L^\infty(G,dg)$.  

The need to compute polynomial integrals for various examples of compact matrix groups arises in many areas of mathematics and physics, including group representation theory, statistical physics, random matrix theory and free probability.  Unfortunately, for most subgroups $G \subseteq U_N$, the evaluation (or even the approximation) of arbitrary polynomial integrals is a non-trivial task.  Even for the most natural examples, such as the orthogonal and unitary groups $O_N$ and $U_N$, the computation of polynomial integrals remains an active area of research.  See for example \cite{AuLa03, Co03, CoSn06, MaNo13}.  

In the context of polynomial integrals over $O_N$ or $U_N$, one is often interested in their behavior in the large $N$ limit.  In this regime, the calculations are simplified by the fact that
the normalized random variables $\{\sqrt{N}u_{ij}\}_{1 \le i,j \le N} \subset L^\infty(O_N,dg)$
(respectively $L^\infty(U_N,dg)$) are asymptotically independent and identically distributed  $N(0,1)$ real (respectively complex) Gaussian random variables.  See for example \cite{DiFr87}.    For example, if $\mathcal G = \{g_{ij}\}_{i,j \in \N}$ denotes an i.i.d. $N(0,1)$ real Gaussian family on some standard probability space $(\Omega, \mu)$ with expectation $\mathbb E = \int_\Omega \cdot d\mu$, then for any (non-commutative) polynomial $P \in \C\langle X_{ij}: i,j \in \N \rangle $, we have \[\lim_{N \to \infty} h_{O_N} \Big(P(\{\sqrt{N}u_{ij}\}_{1 \le i,j \le N})\Big) = \mathbb E(P(\mathcal G)).\]     
In the above equation we of course assume that $N$ is large enough so that the evaluation map $P \mapsto P(\{\sqrt{N}u_{ij}\}_{1 \le i,j \le N}) \in \text{Pol}(O_N)$ is well defined.  The above equality essentially says that for large $N$,
the value of a polynomial integral over $O_N$ or $U_N$ can be approximated by the corresponding (often simpler to compute) Gaussian integral. 

In \cite{Wa}, S. Wang introduced natural non-commutative analogues of the orthogonal  
and unitary groups: the \textit{free orthogonal and free unitary quantum groups} $O_N^+$ and $U_N^+$.  These objects are compact quantum groups (in the sense of S. Woronowicz \cite{Wo2}) and are given in terms of the pairs \[O_N^+ = (C(O_N^+), \Delta_o) \quad \text{and} \quad U_N^+ = (C(U_N^+), \Delta_u),\] where $C(U_N^+)$ is the universal C$^\ast$-algebra generated by $N^2$ elements $\{v_{ij}\}_{1 \le i,j \le N}$ subject to the relations which make the matrices $V = [v_{ij}]$ and $\bar V = [v_{ij}^*]$ unitary in $M_N(C(U_N^+))$, and $C(O_N^+) = C(U_N^+)/\langle v_{ij} = v_{ij}^* \rangle$.  I.e., $C(O_N^+)$ is the universal C$^\ast$-algebra generated by $N^2$ self-adjoint elements $\{u_{ij}\}_{1 \le i,j \le N}$ subject to the relations which make the matrix $U = [u_{ij}]$ unitary.  The \textit{coproduct} $\Delta_u:C(U_N^+) \to C(U_N^+) \otimes_{\text{min}}C(U_N^+)$ is the unique unital C$^\ast$-algebra homomorphism defined by \[\Delta_u(v_{ij}) = \sum_{k=1}^N v_{ik} \otimes v_{kj} \qquad (1 \le i,j \le N),\] giving  $C(U_N^+)$ a (bisimplifiable) C$^\ast$-bialgebra structure.   The formula for the coproduct $\Delta_o:C(O_N^+) \to C(O_N^+) \otimes_{\text{min}}C(O_N^+)$ is the obvious analogue.
Note that the C$^\ast$-algebras $C(O_N^+)$ and $C(U_N^+)$ are free analogues of the commutative function algebras $C(O_N)$ and $C(U_N)$, respectively.

Since $\G = O_N^+, U_N^+$ is a compact quantum group, it admits a \textit{Haar state} $h_\G: C(\G) \to \C$, which is the non-commutative analogue of the left and right translation-invariant Haar measure on a compact group.  More precisely, $h_\G$ is the unique state\footnote{$h_{\G}$ is tracial for $O_N^+$ and $U_N^+$ \cite{Wa}.} satisfying the following $\Delta$-bi-invariance condition \[ (h_\G \otimes \id)  \Delta(a) = (\id \otimes h_\G)\Delta(a) = h_\G(a)1_{C(\G)} \qquad (a \in C(\G)).\] As a consequence of the existence of the Haar state, one can also consider in this non-commutative context the problem of computing polynomial integrals over $O_N^+$ and $U_N^+$, respectively.  More precisely, if $\{u_{ij}\}_{1 \le i,j \le N}$ and $\{v_{ij}\}_{1 \le i,j \le N}$ denote the generators of $C(O_N^+)$ and $C(U_N^+)$, respectively, 
then we want to compute all joint $\ast$-moments \begin{align*}
& h_{O_N^+}(u_{i(1)j(1)}u_{i(2)j(2)} \cdot \ldots \cdot u_{i(k)j(k)}) \qquad (1 \le i(r),j(r)\le N, \ 1 \le r \le k, \ k \in \N),\\
&h_{U_N^+}(v_{i(1)j(1)}^{\epsilon(1)}v_{i(2)j(2)}^{\epsilon(2)} \cdot \ldots \cdot v_{i(k)j(k)}^{\epsilon(k)}) \qquad (1 \le i(r),j(r),\le N, \epsilon(r) \in \{1,\ast\}, \ 1 \le r \le k, \ k \in \N). 
\end{align*}

As in the classical case, the precise computation of these moments is extremely difficult for a fixed dimension $N$.  However, as $N \to \infty$, T. Banica and B. Collins \cite{BaCo} have shown that joint distributions of $\{\sqrt{N}u_{ij}\}_{1 \le i,j \le N}$ and $\{\sqrt{N}v_{ij}\}_{1 \le i,j \le N}$ are modeled by the free probability analogues of real (respectively complex) Gaussian systems.  Before stating their result, we remind the reader of some basic facts about free probability.  For more information on these and related concepts, we refer the reader to the monograph \cite{NiSp}.       

\begin{defn}
\begin{enumerate}
\item A \textit{noncommutative probability space (NCPS)} is a pair $(A,\varphi)$, where $A$ is
a unital C$^\ast$-algebra, and $\varphi: A \to \C$ is a state (i.e. a linear functional such that
$\varphi(1_A) = 1$ and $\varphi(a^*a) \ge 0$ for all $a \in A$).  Elements $a \in A$ are called \textit{random variables}. 
\item Let $(A,\varphi)$ be a NCPS.  A family of $\ast$-subalgebras $\{A_r\}_{r \in \Lambda}$ of $A$ is said to be \textit{freely independent (or free)} if the following condition holds: for any choice of indices $r(1) \ne r(2), r(2) \ne r(3), \ldots , r(k -1 ) \ne r(k) \in \Lambda$ and any choice of centered random variables variables $x_{r(j)} \in A_{r(j)} \in A_{r(j)} \cap \ker \varphi$, we have the equality 
\[\varphi(x_{r(1)}x_{r(2)} \cdot \ldots \cdot x_{r(k)}) = 0.\]   
\item A family of random variables $\{x_r\}_{r \in \Lambda} \subset (A,\varphi)$ is \textit{free} if the family of unital $\ast$-subalgebras
\[\{A_r\}_{r \in \Lambda}; \qquad  A_r := \text{alg}\big( 1,x_r, x_r^*\big),\]
is free in the above sense.
\item A family of random variables $S = \{s_r\}_{r \in \Lambda} \subset (A,\varphi)$ is called a \textit{free semicircular system} if $S$ is free and each $s_r \in S$ is self-adjoint and  identically distributed with respect to $\varphi$ according to Wigner's semicircle law.  That is,
\[\varphi(s_r^k) = \frac{1}{2\pi}\int_{-2}^2 t^k\sqrt{4-t^2}dt \qquad (k \ge 0, \ r \in \Lambda).\] 
\item A family of random variables $C = \{c_{r}\}_{r \in \Lambda} \subset (A, \varphi)$ is called a \textit{free circular system} if $C$ is free and each $c_r \in C$ has the same distribution as the operator $\frac{s_1 +is_2}{\sqrt{2}}$, where $\{s_1,s_2\}$ is a standard free semicircular system.
\item Let $S_N = \{x_r^{(N)}\}_{r \in \Lambda} \subset (A_N,\varphi_N)$ be a sequence of families of random variables and $S = \{x_r\}_{r \in \Lambda} \in (A,\varphi)$ be another family of random variables.  We say that \textit{$S_N$ converges to $S$ (or $S_N \to S$) in distribution as $N \to \infty$} if for
any non-commutative polynomial $P \in \C\langle X_r: r \in \Lambda \rangle$, we have
\[\lim_{N \to \infty} \varphi_N(P(S_N)) = \varphi(P(S)).\] 
\end{enumerate}
\end{defn}

\begin{rem}
\begin{enumerate}
\item Note that (just like classical independence), the joint moments of a freely family $\{x_r\}_{r\in \Lambda} \subset (A,\varphi)$ can be computed from the individual moments of each variable $x_r$ using the definition of freeness.
\item In the context of free probability theory, a free semicircular (circular) system is the free probability analogue of a standard real (complex)) $N(0,1)$-Gaussian family.  
\end{enumerate}
\end{rem}

We are now in a position to state the asymptotic free independence result of T. Banica and B. Collins.

\begin{thm}[\cite{BaCo}, Theorems 6.1 and 9.4] \label{thm:BaCo}
Let $S = \{s_{ij}: i,j \in \N\}$ be a free semicircular system and let $C = \{c_{ij}:  i,j \in \N\}$ be a free circular system in a NCPS $(A,\varphi)$.  For each $N \in \N$, let $S_N = \{\sqrt{N}u_{ij}\}_{1 \le i,j \le N} \subset (C(O_N^+),h_{O_N^+})$ and $C_N = \{\sqrt{N} v_{ij}\}_{1 \le i,j \le N} \subset (C(U_N^+), h_{U_N^+})$ be the normalized generators of $C(O_N^+)$ and $C(U_N^+)$, respectively.  Then 
\[
S_N \to S \quad \text{and} \quad C_N \to C \quad \text{in distribution as $N \to \infty$}.
\]   
I.e., for any $k \in \N$, $1 \le i(r),j(r),\le N$, $\epsilon(r) \in \{1,\ast\}$ and $1 \le r \le k$, we have 
 \begin{align*}
& \lim_{N \to \infty} N^{k/2}h_{O_N^+}(u_{i(1)j(1)}u_{i(2)j(2)} \cdot \ldots \cdot u_{i(k)j(k)}) = \varphi(s_{i(1)j(1)}s_{i(2)j(2)} \cdot \ldots \cdot s_{i(k)j(k)})  \\
&\lim_{N \to \infty}N^{k/2}h_{U_N^+}(v_{i(1)j(1)}^{\epsilon(1)}v_{i(2)j(2)}^{\epsilon(2)} \cdot \ldots \cdot v_{i(k)j(k)}^{\epsilon(k)}) = \varphi(c_{i(1)j(1)}^{\epsilon(1)}c_{i(2)j(2)}^{\epsilon(2)} \cdot \ldots \cdot c_{i(k)j(k)}^{\epsilon(k)}). 
\end{align*}
\end{thm}

Theorem \ref{thm:BaCo} shows that, in an approximate sense, the generators of $O_N^+$ and $U_N^+$ are modeled by free (semi)circular random variables.  In particular, if we denote by $L^\infty(O_N^+)$ and $L^\infty(U_N^+)$ the von Neumann algebras generated by the GNS constructions associated to the corresponding Haar states, then Theorem \ref{thm:BaCo} suggests that these von Neumann algebras may share some analytic properties with the free group factors $L(\F_k)$ ($k \ge 2$).  Over the last decade this has indeed shown to be the case.  See for example \cite{BrAP, Fr13, Is12, VaVe, Ve}.  Unfortunately, the approximation result of Theorem \ref{thm:BaCo} has yet to find a direct application to the study of the analytic structures of $L^\infty(O_N^+)$ and $L^\infty(U_N^+)$.  One reason for this is that the combinatorial ``Weingarten methods'' used to establish this theorem provide very little information about the rate of approximation to a free (semi)circular system.

By exploiting a certain connection between $O_N^+$ and S. Woronowicz's deformed $SU_{-q}(2)$ quantum group \cite{Wo} (where $N = q+q^{-1}$), T. Banica, B. Collins and P. Zinn-Justin  computed the spectral measure of a single generator $u_{ij} \in C(O_N^+)$ relative to the Haar state in \cite{BaCoZJ} .  Using this fact, the authors were able to substantially improve the approximation result of Theorem \ref{thm:BaCo} for a \textit{single generator} $u_{ij}$ of $O_N^+$ by showing that $\sqrt{N}u_{ij}$ \textit{superconverges} (in the sense of H. Bercovici and D. Voiculescu \cite{BeVo}) to a semicircular variable $s \in (A,\varphi)$ .  In particular,
this implies that for any polynomial $P \in \C[X]$, not only do we have that $P(\sqrt{N}u_{ij}) \to P(s)$ in distribution, but we also have the convergence of corresponding operator norms: \[\lim_{N \to \infty} \|P(\sqrt{N}u_{ij})\|_{L^\infty(O_N^+)} = \|P(s)\|_{L^\infty(A,\varphi)}.\]
We note that this norm convergence result is not at all clear from the methods of Theorem \ref{thm:BaCo}.     

The main result of this note is the non-commutative multivariate analogue of the above norm convergence result.  Before stating our main theorem, we recall the notion of strong convergence for random variables, as defined by C. Male in \cite{Ma12}:  Let $(A,\varphi)$ and $(A_N,\varphi_N)$ ($N \in \N$) be NCPS's and let  $S = \{x_r\}_{r \in \Lambda} \in (A,\varphi)$ and $S_N = \{x_r^{(N)}\}_{r \in \Lambda} \subset (A_N,\varphi_N)$ be a sequence
of  families of random variables.  We say that \textit{$S_N \to S$ strongly in distribution} as $N \to \infty$ if $S_N \to S$ in distribution and 
\[\lim_{N \to \infty} \|P(S_N)\|_{L^\infty(A_N,\varphi_N)} =  \|P(S)\|_{L^\infty(A,\varphi)}  \]
for all non-commutative polynomials $P \in \C\langle X_r: r \in \Lambda \rangle$.

Our main result is:

\begin{thm} \label{thm:main}
Let $S = \{s_{ij}:  i,j \in \N\}$ be a standard free semicircular system in a finite von Neumann algebra $(M,\tau)$ and $S_N = \{\sqrt{N}u_{ij}^{(N)}: 1 \le i,j \le N\} \cup \{0:i,j > N\} \subset (C(O_N^+),h_{O_N^+})$ be the (normalized) standard generators of $O_N^+$.  Then \[S_N \longrightarrow S \quad \text{strongly in distribution as $N \to \infty$.}\]
\end{thm} 

The above strong asymptotic freeness result answers a question posed in Section 6 of \cite{BaCuSp} on the mode of convergence to free independence of the joint distribution of the standard generators of $O_N^+$.  This result should also be compared with other recent strong asymptotic freeness results for independent random matrix ensembles.  See for example \cite{HaTh05, CoMa12, Ma12}.

The key ingredient we require for our proof of Theorem \ref{thm:main} is R. Vergnioux's property of rapid decay (property (RD)) for the discrete dual quantum groups $\widehat{O_N^+}$ \cite{Ve}.  We show that property (RD) for the quantum groups  $\{\widehat{O_N^+}\}_{N \ge 3}$ allows us to approximate to any desired degree of accuracy (uniformly in $N$) the operator norm of a fixed non-commutative polynomial in the variables $\{\sqrt{N}u_{ij}^{(N)}\}_{1 \le i,j \le N}$ by its non-commutative $L^p$-norm, for some sufficiently large even integer $q$.  This transfer from $L^\infty$ to $L^p$-norms then allows us to immediately deduce Theorem \ref{thm:main} from Theorem \ref{thm:BaCo}.

The remainder of the paper is organized as follows: In Section \ref{sect:prelim}, we remind the reader of some basic facts on quantum groups, focusing mainly on the example of $O_N^+$.   We then prove Theorem \ref{thm:main} in Section \ref{sect:mainresult} and end with some applications of our result and concluding remarks in Section \ref{sect:applications}.  In particular, we consider the analogous strong convergence result for $U_N^+$ (Corollary \ref{cor:unitary}), strong convergence for polynomials with matrix coefficients (Corollary \ref{cor:matrix_coeff}), and an application to $L^\infty$-$L^2$ norm inequalities for polynomials over free semicircular systems (Corollary \ref{cor:normineq}). 

\subsection*{Acknowledgement}  The author is grateful to Beno\^it Collins for stimulating  discussions which improved an earlier version of this work.  
                     
\section{Preliminaries on $O_N^+$} \label{sect:prelim}

Our main reference for the theory of compact quantum groups will be the book \cite{Ti}.  All unexplained terminology can be found there.  For the remainder of the paper we write $\N = \{1,2,3, \ldots\}$ and $\N_0 = \N \cup \{0\}$. 

\subsection{Peter-Weyl Decomposition of $L^2(O_N^+)$}

Denote by $\Pol(O_N^+) \subset C(O_N^+)$ the dense $\ast$-subalgebra generated by the canonical generators $\{u_{ij}^{(N)}\}_{1 \le i,j \le N} \subset C(O_N^+)$.  We recall that by general compact quantum group theory, the Haar state is always faithful on $\Pol(O_N^+)$.  Denote by $L^2(O_N^+)$ the GNS Hilbert space obtained by completing $\Pol(O_N^+)$ with respect to the sesquilinear form $\langle x |y \rangle = h_{O_N^+}(y^*x)$, and let $L^\infty(O_N^+) \subset \mc B(L^2(O_N^+))$ denote the von Neumann algebra generated by $\Pol(O_N^+)$ acting on $L^2(O_N^+)$ by (extending) left multiplication.  In the following, we simultaneously identify $\Pol(O_N^+)$ with its image as a $\sigma$-weakly dense subalgebra of $L^\infty(O_N^+)$ and as a norm-dense subspace of $L^2(O_N^+)$ via the GNS construction.   

The irreducible representations of the quantum group $O_N^+$ was first studied by T. Banica in \cite{Ba0}, and this gives rise to a natural orthogonal decomposition of $L^2(O_N^+)$ in terms of finite dimensional subspaces spanned by matrix elements of these irreducible representations.

Rather than discussing representations of compact quantum groups, we choose to describe this 
``Peter-Weyl'' decomposition of $L^2(O_N^+)$ intrinsically as follows.  For each $k \in \N_0$, define a subspace $H_k(N)\subset \Pol(O_N^+)$ by setting 
\begin{align*}&H_0(N):=\C 1_{\Pol(O_N^+)}, \quad H_1(N) := \text{span}\{u_{ij}^{(N)}\}_{1 \le i,j \le N}, \\
& H_k(N):=H_1(N)H_{k-1}(N) \ominus H_{k-2}(N) \qquad (k \ge 2),
\end{align*}
where $H_1(N)H_{k-1}(N) := \text{span} \{xy: x \in H_1(N), \ y \in H_{k-1}(N)\}$. The fact that the above recursive definition for $H_k(N)$ makes sense follows from the analysis of the representation theory of $O_N^+$ in \cite{Ba0}.  Moreover, we have the following result.

\begin{thm}[\cite{Ba0}] \label{thm:Banica}
For each $k \in \N_0$, there is a unique (up to isomorphism) unitary representation $U^k$ of $O_N^+$ whose matrix elements span $H_k(N)$.  Conversely, every irreducible unitary representation of $O_N^+$ arises this way.   Moreover, the representations $\{U^k\}_{k \in \N_0}$ satisfy the tensor product decomposition rules
\[U^n \boxtimes U^k \cong \bigoplus_{0 \le r \le \min\{k,n\}} U^{n+k - 2r} \qquad (n,k \in \N_0).\] As a consequence, we have the following orthogonal decomposition
\[L^2(O_N^+) = \ell^2 - \bigoplus_{k \in \N_0} H_k(N),\] and the following multiplication ``rule'' 
\[xy \in \bigoplus_{0 \le r \le \min\{k,n\}} H_{n+k-2r}(N) \qquad (x \in H_n(N), \ y \in H_k(N)).\]
\end{thm}

\subsection{Quantum Numbers}

Fix $0 < q < 1$.  Recall that the \textit{$q$-numbers and $q$-factorials} are defined by the formulas

 \[[a]_q = \frac{q^a-q^{-a}}{q-q^{-1}} = \frac{q^{-a+1}(1-q^{2a})}{1-q^2}, \qquad [a]_q! = [a]_q[a-1]_q \ldots [1]_q  \qquad (a \in \N).\]  Note that as $q \to 1$, $[a]_q \to a$.
   
\subsection{The Property of Rapid Decay}

Using Theorem \ref{thm:Banica}, we can define a length function $\ell:\Pol(O_N^+) = \bigoplus_{k \in \N_0} H_k(N) \to \N_0$ by setting 
\[\ell(x) = \min\Big\{n \in \N_0: x \in \bigoplus_{0 \le k \le n} H_k(N)  \Big\} \qquad (x \in \Pol(O_N^+)).\]  

The main tool we will use to establish the strong asymptotic freeness for the generators of $L^\infty(O_N^+)$ is R. Vergnioux's \textit{property of rapid decay (RD)} \cite{Ve}, which provides a way to estimate the $L^\infty$-norm of any $x \in \Pol(O_N^+)$ in terms of its length $\ell(x)$ and its (much easier to compute) $L^2$-norm.  The main estimate we require is as follows, and should be compared to U. Haagerup's fundamental inequality for free groups (see Lemma 1.3 in \cite{Ha}).    

\begin{thm}[\cite{Ve}] \label{thm:Vergnioux}
For each $l \in \N_0$, let $P_l:L^2(O_N^+) \to H_l(N)$ be the orthogonal projection. 
Then there exists a constant $D_N > 1$ such that for any $n,k,l \in \N_0$, $x \in H_n(N)$ and $y \in H_k(N)$, we have
\[\|P_l(xy)\|_{L^2(O_N^+)} \le D_N \|x\|_{L^2(O_N^+)} \|y\|_{L^2(O_N^+)}.\]
\end{thm}

Combining Theorem \ref{thm:Vergnioux} with the fusion rules from Theorem \ref{thm:Banica}, we obtain an essentially equivalent statement of property (RD) that is in a form more suitable for our purposes.

\begin{cor}[\cite{Ve}] \label{cor:vergnioux}
For each $x \in \Pol(O_N^+)$, we have  
\[\|x\|_{L^2(O_N^+)} \le \|x\|_{L^\infty(O_N^+)} \le D_N(\ell(x) +1)^{3/2}\|x\|_{L^2(O_N^+)}.\]
\end{cor}

\begin{proof}
It suffices to prove that for each $n \in \N_0$, $x \in H_n(N)$, we have $\|x\|_{L^\infty(O_N^+)} \le D_N(n+1)\|x\|_{L^2(O_N^+)}$.  Indeed, since then any $x \in \Pol(O_N^+)$ can be written as $x = \sum_{0 \le n \le \ell(x)} x_n$ where $x_n \in H_n(N)$, which gives  \[\|x\|_{L^\infty(O_N^+)} \le \sum_{0 \le n \le \ell(x)} D_N(n+1)\|x_n\|_{L^2(O_N^+)} \le D_N(\ell(x)+1)^{3/2} \|x\|_{L^2(O_N^+)}.\]  

Now fix $x \in H_n(N)$ and $\xi \in L^2(O_N^+)$.  Then by Theorem \ref{thm:Vergnioux} we have
\begin{align*}
&\|x\xi\|_{L^2(O_N^+)}^2 = \sum_{l \in \N_0}\|P_l(x\xi)\|_{L^2(O_N^+)}^2 = \sum_{l \in \N_0}\Big\|\sum_{k \in \N_0: \ H_l(N) \subset H_n(N)H_k(N)}P_l(x P_k\xi)\Big\|_{L^2(O_N^+)}^2 \\
&\le  \sum_{l \in \N} \Big( \sum_{k \in \N_0: \ H_l(N) \subset H_n(N)H_k(N)} D_N\|x\|_{L^2(O_N^+)}\|P_k\xi\|_{L^2(O_N^+)}\Big)^{2} \\
&\le D_N^2\|x\|_{L^2(O_N^+)}^2\sum_{l \in \N_0}  \#\{k : H_l(N) \subset H_n(N)H_k(N)\} \Big(\sum_{k \in \N_0: \ H_l(N) \subset H_n(N)H_k(N)} \|P_k\xi\|_{L^2(O_N^+)}^2\Big).
\end{align*}
But by Theorem \ref{thm:Banica}, it follows that  $\#\{k \in \N_0 : H_l(N) \subset H_n(N)H_k(N)\}$ and $\#\{l \in \N_0 : H_l(N) \subset H_n(N)H_k(N)\}$ are both bounded by $n+1$.  Therefore 
\begin{align*}
&\|x\xi\|_{L^2(O_N^+)}^2 \le (n+1)D_N^2\|x\|_2^2\sum_{l \in \N_0} \Big(\sum_{k \in \N_0: \ H_l(N) \subset H_n(N)H_k(N)} \|P_k\xi\|_{L^2(O_N^+)}^2\Big)  \\
&\le D_N^2(n+1)^2\|x\|_{L^2(O_N^+)}^2 \sum_{k \in \N_0} \|P_k\xi\|_{L^2(O_N^+)}^2 = D_N^2(n+1)^2\|x\|_{L^2(O_N^+)}^2 \|\xi\|_{L^2(O_N^+)}^2. 
\end{align*}
\end{proof}

\section{Main Result} \label{sect:mainresult}

In this section, we will give a proof of Theorem \ref{thm:main}.  Our strategy will be to first use the property of rapid decay to show that for any fixed non-commutative polynomial $P \in \C\langle X_{ij}: i,j \in \N \rangle$ and any $\epsilon > 0$, there is a fixed $p = p(P,\epsilon) \in (2,\infty)$ such that the non-commutative $L^p$-norm of $P(S_N) \in \Pol(O_N^+)$ is within $\epsilon$ of its $L^\infty$-norm for all $N$.  It is an elementary fact that for each $N$, there is a $p = p(P,\epsilon,N)$ which obtains the required approximation.  They key point here is that we can select a \textit{single} $p \in(2, \infty)$ which works for all $N$.  

Using this uniform $L^p$-$L^\infty$-estimate, we are able to prove Theorem \ref{thm:main} by transferring the problem of norm convergence of polynomials to a question about convergence in distribution.       

\subsection{Uniform $L^p$-$L^\infty$-estimates}

We start with a crucial lemma.

\begin{lem} \label{lem:D_N_bdd}
Let $D_N$ be the constant appearing in Theorem \ref{thm:Vergnioux}.  Then $D_N$ can be chosen so that $\lim_{N \to \infty} D_N = 1$.  In particular, $(D_N)_{N \ge 3}$ is uniformly bounded. 
\end{lem}

\begin{proof}
Fix $n,k,l \in \N_0$, $x \in H_n(N)$ and $y \in H_k(N)$.  If $U^l$ is not equivalent to a subrepresentation of $U^n \boxtimes U^k$, then $P_l(xy) = 0$ and there is nothing to prove.  Otherwise, there exists some $0 \le r \le \min\{k,n\}$ such that $l = n+k - 2r$.  Let $\mc V_l, \mc V_k$ and $\mc V_n$ denote the Hilbert spaces on which the representations $U^l, U^k$ and $U^n$ act, respectively, and let $t_r: \C \cong \mc V_0 \to \mc V_r \otimes \mc  V_r$ be the unique (up to multiplication by $\T$) $O_N^+$-invariant  isometry(uniqueness follows from Theorem \ref{thm:Vergnioux}).  Now identify each representation $U^k$ with the ``highest weight'' subrepresentation of $U^{\boxtimes k}$ in the canonical way (see \cite{Ba0}), and let $p_k:(\C^{N})^{\otimes k} \to \mc V_k$ the orthogonal projection.  Using the map $t_r$ and the projections $p_n,p_k,p_l$ we can define the following canonical $O_N^+$-invariant contraction:  \[\phi_{l}^{n,k} := (p_n \otimes p_k)(\id_{(\C^N)^{\otimes (n-r)}} \otimes t_r \otimes \id_{(\C^N)^{\otimes (k-r)}})p_l: \mc V_l\to \mc V_n \otimes \mc V_k.\]

According to the calculations in Section 4.3 of \cite{Ve} (see Section 4.2 of \cite{BrQAGs} for an analogous calculation for quantum permutation groups), we have
\[\|P_l(xy)\|_{L^\infty(O_N^+)} \le\Big(\frac{\dim \mc V_k\dim \mc V_n}{\dim \mc V_l (\dim \mc V_r)^2}\Big)^{1/2}\frac{\|x\|_{L^2(O_N^+)}\|y\|_{L^2(O_N^+)}}{ \|\phi_{l}^{n,k}\|^2}.\] 

Fix $0 < q < 1$ such that $N = q+q^{-1}$.  Note that $q \to 0$ as $N \to \infty$.  In \cite{Ba0} it is shown that $\dim \mc V_k = [k+1]_q$ for all $k \in \N_0$, so we can take $D_N = \sup_{(n,k,l) \in \N_0^3} \Big(\frac{[k+1]_q[n+1]_q}{[l+1]_q[r+1]_q^2}\Big)^{1/2} \|\phi_{l}^{n,k}\|^{-2}$ and analyze its large $N$ ($\iff$ small $q$) behavior.

As a first observation, note that \[(1-q^2)^3 \le\frac{(1-q^2)(1-q^{2n+2})(1-q^{2k+2})}{(1-q^{2r+2})^2(1-q^{2l+2})} = \frac{[k+1]_q[n+1]_q}{[l+1]_q[r+1]_q^2} 
\le (1-q^2)^{-2}. \] So $\lim_{q \to 0}\sup_{(n,k,l) \in \N_0^3} \Big(\frac{[k+1]_q[n+1]_q}{[l+1]_q[r+1]_q^2}\Big)^{1/2} = 1$ and it suffices to restrict our attention to the quantity $ \|\phi_{l}^{n,k}\|^{-2}$. 

In \cite{Ba0}, T. Banica showed that the monoidal C$^\ast$-tensor category generated by the
fundamental representation $U = [u_{ij}^{(N)}]$ of $O_N^+$ is isomorphic to the Temperley-Lieb category $\TL(q+q^{-1})$.  Through this isomorphism, the linear map $[r+1]_q^{1/2}\phi_{l}^{n,k}$ defined above corresponds to a \textit{three-vertex} with parameters $(n,k,l)$ in $\TL(q+q^{-1})$.  See \cite{KaLi} for the definition of $\TL(q+q^{-1})$ and  three-vertices.  In particular, the results of Section 9.9--9.10 in \cite{KaLi} give an explicit  expression for the norm of a three-vertex with parameters $(n,k,l)$.  Translating this result back to $O_N^+$, we obtain  \begin{align*}
1 &\le  \|\phi_{l}^{n,k}\|^{-2} = \frac{[r+1]_q[l+1]_q![n]_q![k]_q!}{[l+1 +r]_q![n-r]_q![k-r]_q![r]_q!}\\
&=\prod_{s=1}^r \frac{[1+s]_q[n-r+s]_q[k-r+s]_q}{[l+1+s]_q[s]_q^2} \\
&= \prod_{s=1}^r \frac{(1-q^{2+2s})(1-q^{2n-2n+2s})(1-q^{2k-2r+2s})}{(1-q^{2l+2+2s})(1-q^{2s})^2} \\
& \le  \Big(\prod_{s=1}^r \frac{1}{1-q^{2s}}\Big)^3 \le  \Big(\prod_{s=1}^\infty \frac{1}{1-q^{2s}}\Big)^3. 
\end{align*}
Since this last infinite product is finite and converges to $1$ as $q \to 0$, it follows that $\lim_{N \to \infty} D_N = 1$.
\end{proof}

We are now ready to state our result on uniform $L^p$-$L^\infty$ estimates.  In the following, recall that $S_N = \{\sqrt{N}u_{ij}^{(N)}: 1 \le i,j \le N\} \cup \{0: i,j > N\}$ denotes the family of standard normalized generators of $\Pol(O_N^+)$.

\begin{prop} \label{prop:Lpestimate}
Let $P \in \C\langle X_{ij}: i,j \in \N \rangle$ and $\epsilon > 0$.  Then there exists an even integer $p = p(P,\epsilon) \in 2\N$ such that 
\[\|P(S_N)\|_{L^\infty(O_N^+)} \le (1 + \epsilon)\|P(S_N)\|_{L^p(O_N^+)} \qquad (N \ge 3).\]    
\end{prop}

\begin{proof}
Let $r = \deg P$ and $x_N = P(S_N) \in \Pol(O_N^+)$.  Then $\ell(x_N) \le r$ and $\ell((x_N^*x_N)^m) \le 2rm$ for all $m \in \N$.  Indeed, this is an immediate consequence of the fact that $x_N$ and $(x_N^*x_N)^m$ lie in the linear span of the matrix elements of the representations $\{U^{\boxtimes k}\}_{0 \le k \le r}$ and $\{U^{\boxtimes k}\}_{0 \le k \le 2rm}$, respectively.   Fixing $m \in \N$ and applying Corollary \ref{cor:vergnioux}, we then have
\begin{align*}
\|x_N\|_{L^\infty(O_N^+)} &=  \|(x_N^*x_N)^m\|_{L^\infty(O_N^+)}^{1/2m} \le \Big(D_N(\ell((x_N^*x_N)^m) + 1)^{3/2}\|(x_N^*x_N)^m\|_{L^2(O_N^+)}\Big)^{1/2m} \\
&= \Big(D_N(\ell((x_N^*x_N)^m) + 1)^{3/2}\Big)^{1/2m} \|x_N\|_{L^{4m}(O_N^+)}\\
&\le \Big(D_N(2rm+ 1)^{3/2}\Big)^{1/2m} \|x_N\|_{L^{4m}(O_N^+)}. 
\end{align*}
Applying Lemma \ref{lem:D_N_bdd},  we can find an $m = m(r,\epsilon) \in \N$ such that $D_N^{1/2m}(2rm+1)^{3/4m} \le 1 +\epsilon$ for all $N \ge 3$.  Taking $p = 4m$ will then do the job.
\end{proof}

The proof of Theorem \ref{thm:main} now follows easily.

\begin{proof}[Proof of Theorem \ref{thm:main}]
Let $x_N = P(S_N) \in \Pol(O_N^+)$, $x = P(S) \in (M,\tau)$, and $\epsilon >0$. Choose $m \in \N$ large enough so that $\|x\|_{L^{2m}(M)} 
\ge \|x\|_{L^\infty(M)} - \epsilon$.  Applying Theorem \ref{thm:BaCo}, we have \[ \|x\|_{L^{2m}(M)}  = \tau((x^*x)^m)^{1/2m} = \lim_{N \to \infty} h_{O_N^+}((x_N^*x_N)^m)^{1/2m} = \lim_{N \to \infty} \|x_N\|_{L^{4m}(O_N^+)},\] which yields 
\begin{align*}\|x\|_{L^\infty(M)} - \epsilon & \le  \|x\|_{L^{2m}(M)}  = \lim_{N \to \infty} \|x_N\|_{L^{2m}(O_N^+)} \le \liminf_{N \to \infty} \|x_N\|_{L^\infty(O_N^+)}.
\end{align*} 
On the other hand, by Proposition \ref{prop:Lpestimate}, there is a $p = p(P,\epsilon) \in 2\N$ such that  
\begin{align*}
\|x_N\|_{L^\infty(O_N^+)} &\le (1+\epsilon)\|x_N\|_{L^{p}(O_N^+)} \qquad (N \text{ sufficiently large}).
\end{align*}
Applying Theorem \ref{thm:BaCo} once again, we obtain
\[\limsup_{N \to \infty} \|x_N\|_{L^\infty(O_N^+)} \le \limsup_{N \to \infty}(1+\epsilon)\|x_N\|_{L^{p}(O_N^+)} = (1+\epsilon)\|x\|_{L^{p}(M)} \le (1+\epsilon)\|x\|_{L^\infty(M)}.\]
As $\epsilon >0$ was arbitrary, we conclude that $\lim_{N \to \infty} \|x_N\|_{L^\infty(O_N^+)} = \|x\|_{L^\infty(M)}$.
\end{proof}

\section{Some Applications and Consequences} \label{sect:applications}

\subsection{Strong asymptotic freeness for $U_N^+$}

Let $X^{(k)} = \{x^{(k)}_i\}_{i \in I}$ and $Y^{(k)} = \{y^{(k)}_j\}_{j \in I}$ be two sequences of families non-commutative random variables and assume that $(X^{(k)})_{k \in \N}$ and $(Y^{(k)})_{k \in \N}$ converge strongly in distribution to families $X$ and $Y$, respectively.  It has recently been shown by P. Skoufranis \cite{Sk} (see also \cite{Pi12} for an alternate proof) that, if in addition, we assume that the familes $X^{(k)}$ and $Y^{(k)}$ are free for each $k$, then $\{X^{(k)}, Y^{(k)}\} \longrightarrow \{X,Y\}$ strongly in distribution.  

Let $\{v_{ij}^{(N)}\}_{1\le i,j \le N}$ denote the canonical generators of $\Pol(U_N^+)$.  It was shown in \cite{Ba} that the joint distribution $\{v_{ij}^{(N)}\}_{1\le i,j \le N}$ (with respect to the Haar state) can be modeled in terms of the ``free complexification'' of the generators $\{u_{ij}^{(N)}\}_{1\le i,j \le N}$ of $\Pol(O_N^+)$.  More precisely, let $z$ be a  unitary in a NCPS $(A,\varphi)$ whose spectral measure relative to $\varphi$ is the Haar measure on $\T$ (i.e., $z$ is a \textit{Haar unitary}), and assume $z$ is free from $\{u_{ij}^{(N)}\}_{1\le i,j \le N}$.   Then the families \[\{v_{ij}^{(N)}\}_{1\le i,j \le N}  \quad \& \quad \{zu_{ij}^{(N)}\}_{1\le i,j \le N} \]  are identically distributed (i.e., have the same joint $\ast$-moments).

Combining the preceding two paragraphs with Theorem \ref{thm:main} yields the following corollary.  

\begin{cor} \label{cor:unitary}
Let $C = \{c_{ij}:  i,j \in \N\}$ be a standard free circular system in a finite von Neumann algebra $(M,\tau)$ and $\{v_{ij}^{(N)}: 1 \le i,j \le N\}$ be the standard generators of $\Pol(U_N^+)$.  Then \[\{\sqrt{N}v_{ij}^{(N)}: 1 \le i,j \le N\}  \cup \{0: i,j > N\} \longrightarrow C \] strongly in distribution as $N \to \infty$.  
\end{cor}

\subsection{Polynomials with matrix coefficients}

We remark that by a standard ultraproduct technique (see for example Proposition 7.3 in \cite{Ma12}) that the strong convergence of a sequence of families of random variables is equivalent to the (a priori stronger) condition of norm convergence for non-commutative polynomials with matrix coefficients.  

\begin{cor} \label{cor:matrix_coeff}
Let $S = \{s_{ij}:  i,j \in \N\}$ be a standard free semicircular system in a finite von Neumann algebra $(M,\tau)$ and $S_N =\{\sqrt{N} u_{ij}^{(N)}: 1 \le i,j \le N\}$ be the standard normalized generators of $\Pol(O_N^+)$.  Then for any $k \in \N$ and any non-commutative polynomial $P \in M_k(\C) \otimes \C\langle X_{ij}:  i,j \in \N \rangle$, we have \[ \|P(S_N)\|_{M_k(\C) \otimes L^\infty(O_N^+)} \to \|P(S)\|_{M_k(\C) \otimes M}.\] 
\end{cor}

Of course, the analogous result also holds for the generators of $U_N^+$.

\subsection{$L^2$-$L^\infty$ norm-inequalities for free semicircular systems}

By taking limits in Corollary \ref{cor:vergnioux}, and applying Lemma \ref{lem:D_N_bdd} and Theorem \ref{thm:main}, we obtain the following $L^2$-$L^\infty$ norm inequality for a free semicircular system.  This inequality was first obtained by M. Bozejko \cite{Bo91} and is re-proved by P. Biane and R. Speicher using combinatorial methods in \cite{BiSp}.

\begin{cor} \label{cor:normineq}
Let $S = \{s_{i}: i \in I\}$ be a standard free semicircular system in a finite von Neumann algebra $(M,\tau)$ and let $P \in \C\langle X_{i}: i \in I \rangle$. Then 
\[\|P(S)\|_{L^2(M)} \le \|P(S)\|_{L^\infty(M)} \le (\deg P +1)^{3/2} \|P(S)\|_{L^2(M)}.\]
\end{cor}

\subsection{Concluding remarks}
It would be natural to try to adapt the arguments of this paper to establish the strong convergence of certain polynomial functions over other classes of quantum groups, such as S.  Wang's quantum permutation groups $S_N^+$ \cite{Wa98} or the free easy quantum groups studied by T. Banica and R. Speicher in \cite{BaSp09}.  For the quantum permutation groups $S_N^+$, all of the tools are in place: a Weingarten calculus for polynomial integrals over  $S_N^+$ was developed by T. Banica and B. Collins in \cite{BaCo07}, and the author proved a property (RD) result for $S_N^+$ in \cite{BrQAGs}.  Unfortunately, the analogue of Proposition \ref{prop:Lpestimate} that one would require still does not follow immediately.  This is because the constant $D_N$ derived in our property (RD) result for $S_N^+$ turns out to grow \textit{quadratically} in $N$!  It is, however, quite plausible that a more precise analysis in the proof of property (RD) for $S_N^+$ could still yield a uniformly bounded sequence of constants $D_N$ associated to $S_N^+$.


\begin{thebibliography}{99}

\bibitem{AuLa03} S. Aubert and C. S. Lam, {\it Invariant integration over the unitary group},  J. Math. Phys. 44 (2003), 6112--6131.

\bibitem{Ba0} T. Banica, {\it Th\'eorie des repr\'esentations du groupe quantique compact libre $O(n)$},  C. R. Acad. Sci. Paris S\'er. I Math. 322 (1996), no. 3, 241--244.

\bibitem{Ba} T. Banica, {\it Le groupe quantique compact libre $U(n)$}, Comm. Math. Phys. 190 (1997), 143--172.

\bibitem{BaCo} T. Banica and B. Collins, {\it  Integration over compact quantum groups},  Publ. Res. Inst. Math. Sci. 43 (2007), 277--302.

\bibitem{BaCo07} T. Banica and B. Collins, {\it Integration over quantum permutation groups}, J. Funct. Anal. 242 (2007), 641--657. 

\bibitem{BaCoZJ} T. Banica, B. Collins and P. Zinn-Justin, {\it Spectral analysis of the free orthogonal matrix},   Int. Math. Res. Notices 17 (2009), 3286-3309.

\bibitem{BaCuSp} T. Banica, S. Curran and R. Speicher,  {\it De Finetti theorems for easy quantum groups},  Ann. Probab. 40 (2012), 401--435. 

\bibitem{BaSp09} T. Banica and R. Speicher, {\it Liberation of orthogonal Lie groups}, Adv. Math. 222 (2009), 1461--1501.

\bibitem{BeVo} H. Bercovici and D. Voiculescu, { \it Superconvergence to the central limit and failure of the Cram\'er theorem for free random variables}, Probab. Theory Related Fields 103 (1995), 215--222. 

\bibitem{BiSp} P. Biane and R. Speicher, { \it
Stochastic calculus with respect to free Brownian motion and analysis on Wigner space},
Probab. Theory Relat. Fields 112 (1998), 373--409. 

\bibitem{Bo91} M.Bozejko, { \it A q-deformed probability, Nelson's
 inequality and central limit theorems}, Nonlinear fields, classical, random, semiclassical (P. Garbecaki and Z. Popowci, eds.), World Scientific, Singapore (1991), 312--335.

\bibitem{BrQAGs} M. Brannan, {\it Reduced operator algebras of trace-preserving quantum automorphism groups}. Preprint (2012), arXiv:1202.5020.

\bibitem{BrAP} M. Brannan, {\it Approximation properties for free orthogonal and free unitary quantum groups}, J. Reine Angew. Math. 672 (2012), 223--251.

\bibitem{Co03} B. Collins, {\it Moments and Cumulants of Polynomial random variables on unitary groups, the Itzykson-Zuber integral and free probability}, Int. Math. Res. Not. 17 (2003), 953--982. 

\bibitem{CoMa12} B. Collins and C. Male, {\it The strong asymptotic freeness of Haar and deterministic matrices}. Ann. Sci. de l'ENS (2013), to appear.

\bibitem{CoSn06} B. Collins and Piotr Sniady, {\it Integration with respect to the Haar measure on unitary, orthogonal and symplectic group},  Comm. Math. Phys. 264 (2006), 773--795.

\bibitem{DiFr87} P. Diaconis and D. Freedman, { \it A dozen de Finetti-style results in search of a theory}, Ann. Inst. H. Poincare Probab. Statist., 23 (1987), 397--423.

\bibitem{Fr13} A. Freslon, {\it Examples of weakly amenable discrete quantum groups}, J. Funct. Anal. 265 (2013), 2164--2187.

\bibitem{Ha}  U. Haagerup, {\it An example of a nonnuclear C$^{\ast}$-algebra, which has the metric approximation property},  Invent. Math.  50 (1978/79), 279--293.

\bibitem{HaTh05} U. Haagerup and S. Thorbjornson, {\it A new application of random matrices: $\text{Ext}(C^\ast_{\text{red}}(\F_2))$ is not a group}, Ann. Math. 162 (2005), 711--775. 

\bibitem{Is12} Y. Isono, {\it Examples of factors which have no Cartan subalgebras}. Preprint (2012), arXiv:1209.1728.

\bibitem{Ka}  O. Kallenberg, {\it Probabilistic symmetries and invariance principles},  Probability and its applications, Springer-Verlag, (2005).

\bibitem{KaLi} L. Kauffman and S. Lins, {\it Temperley-Lieb recoupling theory and invariants of $3$-manifolds}, Ann. Math. Studies $134$, Princeton University Press, 1994. 

\bibitem{Ma12} C. Male, {\it The norm of polynomials in large random and deterministic matrices}, Probab. Theory Related Fields 154 (2012), 477--532.

\bibitem{MaNo13} S. Matsumoto and J. Novak, { \it Jucys-Murphy elements and unitary matrix integrals}, Int. Math. Res. Not. 2 (2013), 362--397.

\bibitem{NiSp}  A. Nica and R. Speicher, {\it Lectures on the combinatorics of free probability},  London Math. Soc.
Lect. Note Ser. 335, Cambridge University Press, Cambridge 2006.

\bibitem{Pi12} G. Pisier, {\it Remarks on a recent result by Paul Skoufranis}. Preprint (2012), arXiv:1203.3186.  

\bibitem{Sk} P. Skoufranis, {\it On a notion of exactness for reduced free products of C$^\ast$-algebras}, J. Reine Angew. Math. (2013), to appear.

\bibitem{Ti}  T. Timmerman,  {\it An invitation to quantum groups and duality}, EMS Textbooks in Mathematics, Zurich 2008.

\bibitem{VaVe}  S. Vaes and R. Vergnioux, {\it The boundary of universal discrete quantum groups, exactness, and factoriality}, Duke Math. J. 140 (2007), 35--84.

\bibitem{Ve} R. Vergnioux, {\it The property of rapid decay for discrete quantum groups},  J. Operator Theory 57 (2007), 303--324.

\bibitem{Wa} S. Wang, {\it Free products of compact quantum groups},  Comm. Math. Phys. 167 (1995), 671--692.

\bibitem{Wa98} S. Wang, {\it Quantum symmetry groups of finite spaces}, Comm. Math. Phys. 195 (1998), 195--211.

\bibitem{Wo}  S. Woronowicz, {\it Compact matrix pseudogroups},  Comm. Math. Phys. 111 (1987), 613--665.

\bibitem{Wo2} S. Woronowicz,  {\it Compact quantum groups},  Sym\'etries quantiques (Les Houches, 1995), North-
Holland, Amsterdam (1998),  845--884.

\end{thebibliography}
\end{document}